\pdfoutput=1
\documentclass[11pt]{amsart}
\usepackage{amsmath}
\usepackage{amssymb}

\usepackage{enumerate}
\usepackage{slashed}
\usepackage{caption}
\usepackage{subcaption}
\usepackage{newlfont}
\usepackage{amsrefs}
\usepackage{comment}
\usepackage[normalem]{ulem}
\usepackage{accents}
\usepackage{leftidx}

\usepackage{harpoon}
\usepackage[pdftex]{graphicx}
\usepackage{esint}
\usepackage{relsize}
\usepackage[latin1]{inputenc}
\usepackage{xcolor}

\definecolor{brass}{rgb}{0.71, 0.65, 0.36}

\usepackage{tikz-cd}
\theoremstyle{plain}
\newtheorem{theorem}{Theorem}[section]
\newtheorem{lemma}[theorem]{Lemma}
\newtheorem{prop}[theorem]{Proposition}

\theoremstyle{definition}
\newtheorem{remark}[theorem]{Remark}

\numberwithin{equation}{section}

\theoremstyle{plain}

\numberwithin{equation}{section}

\allowdisplaybreaks

\begin{document}

\title[Static Quasi-Local Mass of Tori]{Boundary Behavior of Compact Manifolds With Scalar Curvature Lower Bounds and Static Quasi-Local Mass of Tori}
	
\author[Alaee]{Aghil Alaee}
\address{\parbox{\linewidth}{Aghil Alaee\\
Department of Mathematics and Computer Science, Clark University, Worcester, MA 01610, USA\\
Center of Mathematical Sciences and Applications, Harvard University, Cambridge, MA 02138, USA}}
\email{aalaeekhangha@clarku.edu, aghil.alaee@cmsa.fas.harvard.edu}

\author[Hung]{Pei-Ken Hung}
\address{\parbox{\linewidth}{Pei-Ken Hung\\
School of Mathematics, University of Minnesota, Minneapolis, MN 55455, USA}}
\email{pkhung@umn.edu}

\author[Khuri]{Marcus Khuri}
\address{\parbox{\linewidth}{Marcus Khuri\\
Department of Mathematics, Stony Brook University, Stony Brook, NY 11794, USA}}
\email{marcus.khuri@stonybrook.edu}

\thanks{A. Alaee acknowledges the support of NSF Grant DMS-2316965. M. Khuri acknowledges the support of NSF Grant DMS-2104229.}

\begin{abstract}
A classic result of Shi and Tam states that a 2-sphere of positive Gauss and mean curvature bounding a compact 3-manifold with nonnegative scalar curvature, must have total mean curvature not greater than that of the isometric embedding into Euclidean 3-space, with equality only for domains in this reference manifold.  We generalize this result to 2-tori of Guass curvature greater than $-1$, which bound a compact 3-manifold having scalar curvature not less than $-6$ and at least one other boundary component satisfying a `trapping condition'. The conclusion is that the total weighted mean curvature is not greater than that of an isometric embedding into the Kottler manifold, with equality only for domains in this space. Examples are given to show that the assumption of a secondary boundary component cannot be removed. The result gives a positive mass theorem for the static Brown-York mass of tori, in analogy to the Shi-Tam positivity of the standard Brown-York mass,
and represents the first such quasi-local mass positivity result for non-spherical surfaces. Furthermore, we prove a Penrose-type inequality in this setting.
\end{abstract}
	
\maketitle
	
\section{Introduction}
\label{sec1} \setcounter{equation}{0}
\setcounter{section}{1}

How does a scalar curvature lower bound affect the boundary geometry of a compact Riemannian manifold? In \cite[page 163]{Gromov} Gromov has conjectured that if a compact manifold has scalar curvature lower bound $c\in\mathbb{R}$, then the total mean curvature of the boundary should be bounded above by a constant depending only on $c$ and the intrinsic geometry of the boundary.
Most of the work in this direction has focused on the case when $c=0$ and the boundary consists of spheres, see \cite{SWW} and the references therein. A landmark result of this type for 3-manifolds is that of Shi-Tam \cite{ShiTam} (a version also holds in higher dimensions \cite{EMW}), which shows that if the boundary is mean convex then the total mean curvature of each boundary component having positive Gauss curvature is bounded above by the total mean curvature of its isometric embedding into $\mathbb{R}^3$. The difference of these two total mean curvatures appears naturally in mathematical relativity, and is known as the Brown-York \cite{BrownYork} quasi-local mass 
\begin{equation}
m_{BY}(\Sigma)=\frac{1}{8\pi}\int_{\Sigma}(H_0 -H)dA.
\end{equation}
Here $\Sigma$ is a boundary component of a compact Riemannian 3-manifold $(\Omega,g)$, $H$ is the mean curvature with respect to the outer normal, and $H_0$ is the mean curvature of the isometric image in Euclidean space. Thus, in analogy with the positive mass theorem, their result states that nonnegative scalar curvature of $\Omega$ implies that $m_{BY}(\Sigma)\geq 0$. Moreover, they showed that equality is achieved only when $(\Omega,g)$ is isometric to a domain in $\mathbb{R}^3$. 

One of the major open problems in mathematical relativity is the definition of a universal quasi-local mass for closed spacelike 2-surfaces in spacetime \cite{PenroseR}. Physically this quantity should represent the mass-energy content within the surface, while geometrically it may be interpreted as measuring the deviation of a domain enclosed by the surface from a model geometry. The use of a model or reference spacetime is especially pronounced in the Hamilton-Jacobi family of masses, which have received much attention in recent years particularly with the development of the Wang-Yau mass \cite{WangYau}. The target spacetime is typically taken to be Minkowski space, and in special cases such as with the Brown-York mass the model space is restricted further to the Euclidean constant time slices. Although there are several quasi-local mass definitions of this variety that possess desirable properties for surfaces of spherical topology, including the positive mass theorem and Penrose-type inequalities \cite{AlaeeKhuriYau3,AlaeeLesourdYau,BrownYork,LiuYau,WangYau}, little is known about the quasi-local mass of surfaces with nonzero genus. For instance, the definitions of the Wang-Yau and Liu-Yau masses are inapplicable for tori. On the other hand, Alaee-Khuri-Yau \cite{AlaeeKhuriYau2} recently introduced a notion of mass with positivity properties for surfaces of any topology which admit an isometric embedding into Minkowski space. However, a general existence result for isometric embeddings of non-spherical surfaces into Minkowski space is lacking. Moreover, it is desirable to have a purely Riemannian version of the positivity, \`{a} la Shi-Tam, for surfaces with genus.

In order to expand the range of possible model comparison geometries, it is natural for both mathematical and physical reasons to consider reference spaces which are static. Recall that a \textit{static manifold} is a Riemannian manifold $(M,b)$ which comes endowed with a positive potential function $V$ that lies in the cokernel of the linearized scalar curvature operator
\begin{equation}\label{ofnoinohiw}
(\Delta V) b-\nabla^2 V+V\text{Ric}(b)=0.
\end{equation}
The name is due to the fact that this data and relevant equations arise from static spacetimes which satisfy the Einstein vacuum equations, where the potential function plays the role of a lapse for the timelike Killing field and the Riemannian manifold is a constant time slice. In this setting the static Brown-York mass becomes
\begin{equation}
m_{BY}^S(\Sigma)=\frac{1}{8\pi}\int_{\Sigma}V\left(H_0-H\right) dA,
\end{equation}
where $H_0$ is the mean curvature of the isometric embedding $\Sigma\hookrightarrow (M,b)$. Positivity properties of the static Brown-York mass for spheres with respect to the Schwarzschild manifold have been established by Lu-Miao \cite{LuMiao} with a localized Penrose inequality, and a related spacetime version of this was obtained by Alaee-Khuri-Yau \cite{AlaeeKhuriYau3}. Furthermore, when the reference manifold is hyperbolic space Shi-Tam \cite{ShiTam2007} and Wang-Yau \cite{WangYau2007} have proven the positive mass theorem for the static mass of spheres. Other developments concerning static adaptions of various quasi-local masses have been found by several authors, for a survey see \cite{Wang}.

In this paper, we prove the first positive mass theorem for the static quasi-local mass of tori. The model geometry will be taken to be the static Kottler manifold $(\mathbb{R}\times T^2,b=ds^2+e^{2s}\sigma)$, where $\sigma$ is a standard flat metric on the 2-torus. The potential function $V=e^s$ ensures that the static equations \eqref{ofnoinohiw} are satisfied.
Note that this manifold arises as a quotient of hyperbolic space $\mathbb{H}^3$ with identifications along horospheres, and is the induced metric on a constant time slice of the (toroidal) Kottler spacetime \cite{ChruscielSimon} with zero mass and cosmological constant $\Lambda = -3$.
A primary motivation for this choice of reference space is the isometric embedding result of Gromov \cite[Theorem $B'$ (iii), Section 3.2.4]{Gromov1}, which shows that any metric on the torus can be isometrically embedded in a scaled Kottler manifold with sufficiently large negative curvature. 
Alternatively, one may appropriately scale any torus to obtain an isometric embedding into the fixed Kottler manifold with scalar curvature $R=-6$. These isometric embeddings are constructed as graphs over the $s=0$ level set; we shall refer to such maps as \textit{graphical isometric embeddings}. It follows that there is a large family of tori which satisfy the isometric embedding criteria of the main theorem below.

In order to state the result, some restrictions on the topology of the domain $(\Omega,g)$ and its relation to the surface of interest $\Sigma$, will be needed.
In particular, we will utilize the so called homotopy condition from \cite{EGM}, which generalizes the
situation in which there is a retraction of $\Omega$ onto its boundary component $\Sigma$. The manifold $\Omega$ will be said to satisfy the \textit{homotopy condition} with respect $\Sigma$, if there exists a continuous map
$\rho:\Omega\rightarrow\Sigma$ such that its composition with inclusion $\rho\circ i:\Sigma\rightarrow\Sigma$
is homotopic to the identity.

\begin{theorem}\label{thm:PMT} 
Let $(\Omega,g)$ be a compact connected and orientable smooth Riemannian 3-manifold with boundary $\partial\Omega=\Sigma \sqcup \Sigma_h$, satisfying scalar curvature lower bound $R\geq -6$, the homotopy condition with respect to $\Sigma$, and $H_2(\Omega,\Sigma_h;\mathbb{Z})=0$.  Assume that $\Sigma$ is a torus with Gauss curvature $K>-1$ and mean curvature $H>0$ with respect to the outer normal, and that it admits a graphical isometric embedding into a Kottler manifold.
\begin{enumerate}
\item If the boundary $\Sigma_h$ is nonempty and has mean curvature $H\leq 2$ with respect to the inner normal, then $m_{BY}^S(\Sigma)\geq 0$.

\item The same conclusion holds if the boundary $\Sigma_h$ contains additional components of genus zero with $H\leq -2$.

\item Under the same hypotheses, if $m_{BY}^S(\Sigma)= 0$ then $\Sigma_h$ has one component and $(\Omega,g)$ is isometric to a region in a Kottler manifold.

\item If $\Sigma_h$ is a minimal surface and $\Sigma_h^1$ is the component with least area then a Penrose-type inequality holds
\begin{equation}\label{thm1.2}
m_{BY}^S(\Sigma)\geq \mathcal{C}\frac{|\Sigma_h^1|}{16\pi},
\end{equation}
where $\mathcal{C}=4\min_{\Sigma_h}\partial_{\upsilon}u>0$ in which $u$ a particular spacetime harmonic function and $\upsilon$ denotes the unit inner normal.
\end{enumerate}
\end{theorem}

\begin{remark}
These results fail to hold if the auxiliary boundary component $\Sigma_h$ does not satisfy the mean curvature condition, or $\Omega$ does not satisfy the homotopy condition. Explicit examples are given in the appendix to illustrate this point.
\end{remark}

The proof will rely on a Shi-Tam style extension argument, together with the so called \textit{spacetime harmonic functions} that satisfy the equation $\Delta u=3|\nabla u|$ and which were introduced in the context of the asymptotically flat spacetime positive mass theorem \cite{HKK}.
These functions were later employed to treat the asymptotically locally hyperbolic case of a toroidal infinity by the current authors \cite{AlaeeKhuriHung}, and Theorem \ref{thm:PMT} may be viewed as a localization of those results analogous to the localizations of global positive mass theorems obtained in \cite{AlaeeKhuriYau3,LuMiao,ShiTam,ShiTam2007,WangYau2007}. The rigidity statement in Theorem \ref{thm:PMT} is also related to those of Eichmair-Galloway-Mendes \cite[Corollary 1.4]{EGM} and Yau \cite[Theorem 3.2]{Yau}, who obtained the same conclusion with different hypotheses involving a pointwise inequality for the mean curvature of $\Sigma$ instead of the quasi-local mass condition. Note that the pointwise assumption on the mean curvature of $\Sigma_h$ can be interpreted as a \textit{trapping condition}, whereby in a spacetime setting with umbilic initial data the null-mean curvature of this surface is nonpositive, signaling a trapped surface that is associated with black holes. 

This paper is outlined as follows. In section \ref{sec:flow}, we prove global existence and asymptotics of the unit normal flow within the Kottler manifold emanating from the isometric embedding of $\Sigma$. Section \ref{Sec:BSTextension} is dedicated to the construction of a Bartnik-Shi-Tam asymptotically hyperbolic extension with toroidal infinity, while in Section \ref{Sec:asymptotes} it is shown that the static Brown-York mass of the extension's leaves converges to the total mass. A positive mass theorem with corners for asymptotically hyperbolic manifolds with toroidal infinity is then established with the aid of spacetime harmonic functions, and Theorem \ref{thm:PMT} is proved in the final Section \ref{Sec:mainthm}. Lastly, an Appendix \ref{appA} is included to provide counterexamples to the conclusions of the main theorem if the hypothesis of nonempty $\Sigma_h$ is removed.

\section{Unit Normal Flow }\label{sec:flow}
\setcounter{equation}{0}
\setcounter{section}{2}

In this section we consider the asymptotics and convergence of the unit normal flow in a Kottler manifold, emanating from a graphical surface with positive semidefinite second fundamental form. Let $M_k=(\mathbb{R}\times T^2 , b=ds^2 +e^{2s}\sigma)$ be a Kottler manifold with flat metric $\sigma=\sigma_{ij}d\theta^i d\theta^j$ on the 2-torus where $\sigma_{ij}$ are constants and $\theta^i$ are circle coordinates, and consider a graph over the $s$-level sets
\begin{equation}
\Sigma_0=\{ (v(\theta),\theta)\, |\, \theta=(\theta^1,\theta^2)\in T^2 \}
\end{equation}
given by a smooth function $v$ on $T^2$.
Then the induced metric on $\Sigma_0$  as well as its inverse are given by
\begin{equation}\label{gamma1}
\gamma_{ij}=e^{2v}\sigma_{ij}+v_i v_j,\quad\quad\quad \gamma^{ij}=e^{-2v}\sigma^{ij}-\rho^{-2}
e^{-4v} v^i v^j,
\end{equation}
where $v_i=\partial_{\theta^i}v$, $v^i =\sigma^{ij}v_j$, and
$\rho^2=1+e^{-2v}|\nabla v|_{\sigma}^2$. Moreover,
the unit normal vector to $\Sigma_0$ pointing towards infinity is
\begin{equation}
\nu=\rho^{-1}\left( \partial_s-e^{-2v}v^i\partial_{\theta^i} \right),
\end{equation}
while the second fundamental form may be expressed as
\begin{equation}\label{equ:2ndff}
h_{ij}=b(\pmb{\nabla}_i \text{ }\!\nu,\partial_{\theta^j})=\rho^{-1}\left( -\nabla_{ij}v+2v_i v_j+e^{2v}\sigma_{ij} \right)
\end{equation}
where $\pmb\nabla$ and $\nabla$ are the Levi-Civita connections with respect to $b$ and $\sigma$, respectively. 
Let $F:[0,t_0)\times T^2\to M_k$ be a 1-parameter family of graphical embeddings with $F(t,\bar{\theta})=(v(t,\Theta(t,\bar{\theta})),\Theta(t,\bar{\theta}))$ which satisfies the unit normal flow equation
\begin{equation}\label{equ:unit_flow}
\partial_t F=\nu,
\end{equation}
where $\Theta(t,\cdot):T^2 \rightarrow T^2$ is a 1-parameter family of diffeomorphisms.
We may write $\Sigma_t=F(t,T^2)$, and for convenience will often continue to denote quantities relevant to $\Sigma_t$ as above without reference to $t$.

\begin{prop}\label{pro:longtime}
Suppose that the second fundamental form of the initial surface $\Sigma_0 \subset M_k$ is positive semidefinite. Then the unit normal flow exists for all time as a smooth graph, and the translated graph functions $v(t,\cdot)-t$ converge in $C^{\ell}(T^2)$ for any $\ell$ to a smooth function on the torus. Moreover, the maps $\Theta(t,\cdot)$ converge smoothly to a limiting diffeomorphism.
\end{prop}

\begin{proof}
First note that local in time existence follows from a standard implicit function theorem argument for this flow. To obtain long time existence we will establish uniform a priori estimates on a maximal time interval $[0,t_0)$. Notice that the evolution equation \eqref{equ:unit_flow} is equivalent to $\dot{v}=\rho^{-1}$, and $\dot{\theta}^i=-\rho^{-1}e^{-2v}v^i$ where dot denotes differentiation in time, and thus
\begin{equation}\label{equ:unit_flow_graph}
\partial_t v=\dot{v}-\dot{\theta}^i \partial_{\theta^i}v=\rho=\left(1+e^{-2v}|\nabla v|_{\sigma}^2\right)^{1/2}.
\end{equation}
This shows that $v$ is nondecreasing, and that $v-t$ satisfies the transport equation
\begin{equation}
\partial_t (v-t)=(\rho+1)^{-1}e^{-2v}v^i \partial_{\theta^i}(v-t).
\end{equation}
Integrating along characteristics yields an upper bound while the monotonicity gives a lower bound, namely
\begin{equation}\label{equ:C0}
v_{\min}(0)\leq v_{\min}(t)\leq  v_{\max}(t)\leq v_{\max}(0)+t,
\end{equation}
where $v_{\max}(t)=\max_{\theta\in T^2} v(t,\theta)$ and $v_{\min}(t)=\min_{\theta\in T^2} v(t,\theta)$. Furthermore, control on
first derivatives may be obtained from the computation
\begin{align}
\begin{split}
\partial_t(\rho^2-1)=&-2(\partial_t v)e^{-2v}|\nabla v|_{\sigma}^2 +2e^{-2v}v^i\partial_{\theta^i}(\partial_t v)\\
=&-2\rho(\rho^2-1)-\rho^{-3}e^{-2v}v^i\partial_{\theta^i}(\rho^2 -1),
\end{split}
\end{align}
where \eqref{equ:unit_flow_graph} has been used. 
More precisely, viewing this as a transport equation
for $\rho^2 -1$ and integrating again along characteristics while utilizing $-2\rho(\rho^2-1)\leq -2(\rho^2- 1)$ produces
\begin{equation}\label{oingoqinoginqh}
\rho^2 -1\leq Ce^{-2t}
\end{equation}
where the constant $C$ depends only on $\rho|_{\Sigma_0}$.
Therefore \eqref{equ:C0} yields
\begin{equation}\label{equ:C1}
|\nabla v|_{\sigma}^2 =e^{2v}(\rho^2 -1)\leq Ce^{2v_{\max}(0)}.
\end{equation}
	
To obtain $C^2$ bounds, we will employ the second fundamental form. Observe that since the ambient metric has constant sectional curvature $-1$, the relevant evolution equation is
\begin{equation}\label{oinfoiqnh}
\partial_t h_{i}^j=-h^k_i h^j_k+\delta_i^j.
\end{equation} 
This equation admits an explicit solution, which when interpreting the $(1,1)$-tensors as matrices may be expressed as
\begin{equation}\label{oinqoinoiqnh}
h(t)=\delta+2e^{-2t}(h(0)-\delta)\left[h(0)+\delta-(h(0)-\delta)e^{-2t}\right]^{-1}.
\end{equation}
Since the initial second fundamental form $h(0)$ is positive semidefinite (with respect to $\gamma(0)$), we find that $h(t)$ inherits this property for all future times, moreover
\begin{equation}\label{hgamma}
|h-\gamma|_\gamma \leq Ce^{-2t}.
\end{equation}
Next note that \eqref{gamma1} and \eqref{equ:2ndff} show that
\begin{equation}\label{foniqoinoinqh}
\nabla_{ij}v=-\rho h_{ij}+2v_i v_j +e^{2v}\sigma_{ij}
=-\rho(h_{ij}-\gamma_{ij})+(2-\rho)v_i v_j +(1-\rho)e^{2v}\sigma_{ij}.
\end{equation}
Hence \eqref{equ:C0}, \eqref{equ:C1}, and \eqref{hgamma} imply that
\begin{equation}
|\nabla^2 v|_{\sigma}\leq C,
\end{equation}
where $C$ depends only on $\Sigma_0$.

In order to estimate higher order derivatives, let $D$ be the Levi-Civita connection of $\gamma$ and set $\beta_{ij}=h_{ij}- \gamma_{ij}$.
Then we claim that
\begin{equation}\label{lem:beta}
|D^m \beta|_\gamma\leq C_m e^{-(2+m)t}
\end{equation}
for any integer $m\geq 0$, where the constants $C_m$ only have dependence on $m$ and $\Sigma_0$. To see this, first note that the case $m=0$ has already been established in \eqref{hgamma}.
Proceeding by induction, assume that the estimate holds for integers up to and including $m-1$. Observe that \eqref{oinfoiqnh} together with $\partial_t \gamma_{ij}=2h_{ij}$ produces
\begin{equation}
\partial_t \beta_{ij}=\beta_{ik}\beta^k_j,
\end{equation}
and a direct computation gives
\begin{equation}
\partial_t \gamma( D^m \beta,D^m \beta)=\left(\partial_t \gamma\right)( D^m \beta,D^m \beta)+2\gamma\left( \left[ \partial_t, D^m \right]\beta, D^m\beta \right)+2\gamma\left( D^m\partial_t\beta, D^m\beta \right). 
\end{equation}
Each term from this expression will be analyzed separately. In particular, using $\partial_t \gamma^{ij}=-2\gamma^{ij}-2\beta^{ij}$ yields
\begin{equation}
\left(\partial_t \gamma\right)( D^m \beta,D^m \beta)=-2(m+2)\gamma ( D^m \beta,D^m \beta)+\beta\ast D^m \beta\ast D^m \beta.
\end{equation}
Furthermore, since the time derivative of Christoffel symbols for $\gamma$ takes the form
\begin{equation}
\partial_t \Gamma_{ij}^k =\gamma^{kl}\left(D_i h_{jk} +D_j h_{ik}-D_l h_{ij}\right)
=\gamma^{kl}\left(D_i \beta_{jk} +D_j \beta_{ik}-D_l \beta_{ij}\right),
\end{equation}
we find that
\begin{equation}
\left[ \partial_t, D^m \right]\beta=\sum_{p=0}^{m-1}D^{m-p}\beta\ast D^{p}\beta,
\end{equation}
and hence
\begin{equation}
\gamma\left( \left[ \partial_t, D^m \right]\beta, D^m\beta \right)=\sum_{p=0}^{m-1}D^{m-p}\beta\ast D^{p}\beta\ast D^m\beta. 
\end{equation} 
Next, observe that the evolution equation for $\beta$ shows
\begin{equation}
\gamma\left( D^m\partial_t \beta, D^m\beta \right)=\sum_{p=0}^{m}D^{m-p}\beta\ast D^p\beta\ast D^m\beta,
\end{equation}
and therefore by combining the above we find
\begin{equation}
\partial_t \gamma( D^m \beta,D^m \beta)=-2(m+2)\gamma ( D^m \beta,D^m \beta)+\beta\ast D^m \beta\ast D^m \beta +\sum_{p=1}^{m-1}D^{m-p}\beta\ast D^p\beta\ast D^m\beta.
\end{equation}
It then follows from the induction hypothesis that
\begin{equation}
\partial_t |D^m \beta|_\gamma^2 \leq ( -2(m+2)+C e^{-2t})|D^m \beta|_\gamma^2 +C e^{-2(m+4)t} |D^m\beta|_\gamma,
\end{equation}
where $C$ is a constant depending only on $m$. Treating this as a linear first order ODE inequality for $|D^m \beta|_{\gamma}$ immediately gives rise to the desired estimate \eqref{lem:beta}, and the formula \eqref{foniqoinoinqh} then yields
\begin{equation}\label{aoifnoinoih}
|\nabla^{m+2} v|_{\sigma}\leq C_m.
\end{equation}

We may now bootstrap equation \eqref{equ:unit_flow_graph} to obtain uniform bounds for derivatives in time and space
\begin{equation}
|\partial_t^l \nabla^m v|_{\sigma}\leq C_{lm},
\end{equation}
for some $C_{lm}$ dependent only on $l$, $m$, and $\Sigma_0$. Standard arguments then show that the flow may be extended smoothly beyond the maximal time $t_0$, and therefore the solution exists for all time.

It remains to establish convergence	of the translated graphs. Notice that \eqref{oingoqinoginqh} implies
\begin{equation}
|\partial_t (v-t)|=\rho-1\leq Ce^{-2t}.
\end{equation}
It follows that $v(t,\cdot)-t$ is Cauchy in $C^0(T^2)$ as $t\rightarrow\infty$, and therefore converges to a continuous limit. Furthermore, differentiating the equation \eqref{oingoqinoginqh} and using \eqref{aoifnoinoih} leads to
\begin{equation}
|\partial_t \nabla^m (v-t)|_{\sigma}\leq C_m e^{-2t},
\end{equation}
from which we obtain convergence of all higher order derivatives.

Lastly, consider the second initial value problem defining the normal flow
\begin{equation}\label{oanoinoih}
\partial_t\Theta=-\rho^{-1}e^{-2v}\nabla v=: \mathcal{U}(t,\Theta),\quad\quad \Theta(0,\cdot)=\mathrm{Id},
\end{equation}
where $\mathrm{Id}:T^2 \rightarrow T^2$ is the identity map. Since the vector fields $\mathcal{U}$ are smooth and the torus is compact, a unique solution consisting of a 1-parameter family of diffeomorphisms exists for all time, and is smooth in both time and space. Note that by the estimates above for the graph function $v$, we have $|\mathcal{U}|_{\sigma}\leq Ce^{-2t}$ globally for some constant $C$. It follows from \eqref{oanoinoih} that $\Theta(t,\cdot)$ is Cauchy in $C^{0}(T^2,T^2)$ as $t\rightarrow\infty$, and thus converges to a continuous limiting map. To obtain higher order convergence, differentiate the evolution equation to find
\begin{equation}\label{fohqaoinoinhoi}
\partial_t\left(\partial_{\bar{\theta}^i}\Theta\right)=\left(\nabla_{\partial_{\theta^j}}\mathcal{U}\right)\left(\partial_{\bar{\theta}^i}\Theta^j \right).
\end{equation}
Viewing this as a first order linear ODE for the vector field $\partial_{\bar{\theta}^i}\Theta$, in which the coefficients satisfy $|\nabla\mathcal{U}|_{\sigma}\leq Ce^{-2t}$, we find that this vector field is uniformly bounded in time. Using this fact in the equation \eqref{fohqaoinoinhoi} then implies that $\Theta(t,\cdot)$ is Cauchy in $C^{1}(T^2,T^2)$ as $t\rightarrow\infty$, and thus converges to a continuously differentiable limiting map $\Theta_{\infty}$. Continuing in this way with a bootstrap argument yields the desired smooth convergence. In order to show that the limiting map is a diffeomorphism, first observe that homotopy invariance of the degree implies that $\mathrm{deg} \text{ }\!\Theta_{\infty}=1$, and hence it is surjective. We next claim that this map is a local diffeomorphism. Assume by way of contradiction that the Jacobian determinant is zero at a point $\bar{\theta}\in T^2$, then there exists a nonzero pair of constants such that $\Sigma_i c^i\partial_{\bar{\theta}^i}\Theta_{\infty}(\bar{\theta})=0$. Consider the time dependent vectors at this point $X(t)=\Sigma_i c^i\partial_{\bar{\theta}^i}\Theta(t,\bar{\theta})$, and note that \eqref{fohqaoinoinhoi} implies
\begin{equation}
|\partial_t \log |X|_{\sigma}|\leq Ce^{-2t},
\end{equation}
which yields a contradiction for sufficiently large $t$. We conclude that the Jacobian determinant is always nonzero for the limiting map, and hence by the inverse function theorem it is a local diffeomorphism. Because $\Theta_{\infty}:T^2 \rightarrow T^2$ is also proper, it is a covering map. Furthermore, since $\Theta(t,\cdot)$ and $\Theta_{\infty}$ are homotopic, after conjugation they induce the same homomorphism between fundamental groups, so that their respective image subgroups have the same index. It follows that the degree of the cover is 1 which implies $\Theta_{\infty}$ is injective, and therefore is a diffeomorphism.
\end{proof}
	
\section{A Batrnik-Shi-Tam Extension}\label{Sec:BSTextension}
\setcounter{equation}{0}
\setcounter{section}{3}	

In this section we will construct an asymptotically hyperbolic extension of the compact manifold $\Omega$ by utilizing the isometric embedding of its toroidal boundary component
$\Sigma$ into the Kottler manifold, as well as the unit normal flow from the previous section. This is analogous to the extensions employed by Bartnik \cite{Bartnik} and Shi-Tam \cite{ShiTam, ShiTam2007}
in the case of spheres. Let $\{\Sigma_t\}_{t\in[0,\infty)}$ be the leaves of the unit normal flow emanating from the isometric image of $\Sigma$, and denote the corresponding induced metrics, Gauss curvatures, mean curvatures, second fundamental forms, and principal curvatures by $\gamma_t$, $K_t$, $H_t$, $h_t$, and $\lambda_1(t)$, $\lambda_2(t)$. 
Since the Kottler metric $b$ is of constant curvature we have $\mathrm{Ric}_b(\nu,\nu)=-2$ where $\nu$ is the unit normal to $\Sigma_t$, and therefore taking two traces of the Gauss equations yields
\begin{equation}\label{eqn.1.1}
2\lambda_1(t)\lambda_2(t)=H_t^2-|h_t|_{\gamma_t}^2=2\text{Ric}_{b}(\nu,\nu)+(2K_t+6)= 2K_t+2.
\end{equation}
It follows that if $K_0 > -1$ then $\Sigma_0$ has positive definite second fundamental form, since both $\lambda_1(0)$ and $\lambda_2(0)$ cannot be negative as this would lead to a contradiction with \eqref{equ:2ndff} by taking a trace and applying the maximum principle. Proposition \ref{pro:longtime} then shows that the flow exists for all time, and \eqref{oinqoinoiqnh} guarantees that the principal curvatures remain positive throughout, which in turn preserves the Gauss curvature bound $K_t >-1$.

The normal flow produces a foliation of $M_{k}$ outside of $\Sigma_0$. 
This region in parametric space will be denoted by $\Omega_+=F^{-1}\left(\cup_{t=0}^\infty\Sigma_t\right)$, and the Kottler metric there may be expressed as $F^*b=dt^2+\gamma_t$.
Then a Bartnik-Shi-Tam extension of $\Omega$ is the Riemannian manifold $(\Omega_+,g_+)$ with
\begin{equation}\label{gplus}
g_+=w^2dt^2+\gamma_t,
\end{equation} 
where the function $w$ satisfies the evolution equation
\begin{equation}\label{PDEu}
\begin{cases}
H_t \partial_t w=w^2\Delta_t w+\left(K_t+3\right)\left(w-w^3\right) & \text{ on $\Omega_+$}\\
w(0,\bar{\theta})=w_0(\bar{\theta})>0
\end{cases},
\end{equation}
in which $\Delta_t$ is the Laplace-Beltrami operator for $\gamma_t$.
This equation guarantees \cite[(2.10)]{WangYau2007} that the scalar curvature of the extension is $R_{g_+}=-6$.

\begin{lemma}\label{ogoinoinhoih}
Suppose that the surface $\Sigma_0\subset M_k$ admits the Gauss curvature lower bound $K_0 >-1$. 
Then there exists a unique smooth positive solution to \eqref{PDEu} for all time, and
\begin{equation}\label{limit1}
\lim_{t\to\infty}e^{3t}\left(w-1\right)=\mathbf{w}_{\infty}
\end{equation}
for some $\mathbf{w}_{\infty}\in C^{\infty}(T^2)$. Moreover, this convergence occurs in $C^{\ell}(T^2)$ for any $\ell$, and time derivatives of this quantity decay at the rate $e^{-2t}$.
\end{lemma}
	
\begin{proof} 
As noted above, the Gauss curvature lower bound guarantees that the normal flow exists for all time and has leaves with positive definite second fundamental form, in particular $H_t >0$. Moreover, \eqref{hgamma} implies that $H_t=2+O(e^{-2t})$ as $t\rightarrow\infty$, and thus the mean curvature has a uniformly positive lower bound throughout the flow. Since the initial condition $w_0>0$, the linearized equation for \eqref{PDEu} is then strictly parabolic so that the implicit function theorem guarantees the existence of a unique positive solution $w$ for a short time. To extend the flow further, uniform $C^0$ bounds will be applied. Namely, following \cite[Lemma 2.2]{ShiTam} one may solve the ODE
\begin{equation}
\overline{w}'(t)=\eta(t)(\overline{w}-\overline{w}^3)
\end{equation}
to obtain
\begin{equation}
\overline{w}[C,\eta](t)=\left(1+C e^{-2\int_{0}^{t}\eta(\bar{t})d\bar{t}}\right)^{-1/2},
\end{equation}
where $C$ is a constant.  A comparison argument then yields
$\overline{w}[C_{-},\eta_-]\leq w\leq\overline{w}[C_{+},\eta_{+}]$ in which
\begin{equation}
\eta_{+}=\min_{\Sigma_t}\frac{K_t+3}{H_t},\quad\quad C_+=-1+\left(\max_{\Sigma_0} w_0 +1\right)^{-2},
\end{equation}
and if $\min_{\Sigma_0}w_0 >1$ then
\begin{equation}
\eta_{-}=\max_{\Sigma_t}\frac{K_t+3}{H_t},\quad\quad C_-=\left(\min_{\Sigma_0} w_0 \right)^{-2} -1,
\end{equation}
otherwise replace $\eta_-$ with $\eta_+$ to achieve the lower bound. Notice that \eqref{hgamma} gives $\lambda_i=1+O(e^{-2t})$, $i=1,2$ and from \eqref{eqn.1.1} this implies $K_t =O(e^{-2t})$. Therefore $\eta_{\pm}=3/2+O(e^{-2t})$, and the barrier estimates show that $w$ remains uniformly positive while also satisfying the asymptotics
\begin{equation}\label{C0}
|w-1|\leq C_0 e^{-3t},
\end{equation}
for some constant $C_0$. This analysis may be combined with Krylov-Safanov estimates \cite[Theorem 4.2]{KS} to obtain control in $C^{0,\alpha}$, followed by Schauder estimates to bound all higher derivatives. Long time existence and uniqueness of a smooth solution to \eqref{PDEu} is then established.

In order to obtain refined asymptotics, consider the function $z=e^{3t}\left(w-1\right)$. By rescaling the metric $\hat{\gamma}_t=e^{-2t}\gamma_t$ on $\Sigma_t$, we find that this function satisfies
\begin{align}
\begin{split}
\partial_t z =&e^{3t}\partial_t w+3z\\
=&e^{3t}\left(H_t^{-1}w^2 \Delta_t w+H_t^{-1}(K_t +3)(w-w^3)\right)+3z\\
=&e^{-2t}H_t^{-1}w^2\hat{\Delta}_t z+\left(3-H_t^{-1}(K_t +3)w(1+w)\right)z.\\
\end{split}
\end{align}
Here $\hat{\Delta}_t$ is the Laplacian with respect to $\hat{\gamma}_t$, which according to \eqref{gamma1} and Proposition \ref{pro:longtime} is uniformly equivalent to the Laplace operator with respect to the flat metric $\sigma$ on $T^2$. This equation suggests the change of variables $\tau=-\frac{1}{2}e^{-2t}$, which produces
\begin{equation}
\partial_\tau z=H_t^{-1}w^2\hat{\Delta}_tz+e^{2t}\left(3-H_t^{-1}(K_t +3)w(1+w)\right)z
\end{equation}
for $\tau\in[-\tfrac{1}{2},0)$. Note that the asymptotics discussed above show that the zeroth-order coefficient is $O(1)$ are $\tau\rightarrow 0$. Therefore, since the equation is uniformly parabolic standard theory shows that $z$ extends smoothly to a function $\mathbf{w}_{\infty}\in C^{\infty}(T^2)$ at $\tau=0$. The desired limit \eqref{limit1} and convergence properties now follow.
\end{proof}

Having constructed the extension $(\Omega_+,g_+)$ of $(\Omega,g)$, we will now verify that it is asymptotically hyperbolic; the precise definition of this type of asymptotic condition is provided in the next section. To accomplish this goal, we observe that in addition to the original Kottler coordinates
$(s,\theta)$ on $F(\Omega_+)\subset M_k$, an alternate set of coordinates $(t,\theta)$ is also available arising from the foliation given by the normal flow. Furthermore, the relation between the two sets of coordinates is provided by Proposition \ref{pro:longtime} and its proof, namely 
\begin{equation}\label{qoignoqioginqh}
s(t,\theta)-t=f(\theta)+O(e^{-2t})
\end{equation}
for some function $f\in C^{\infty}(T^2)$ as $t\to\infty$, where all derivatives of this quantity also decay at the same rate when measured with respect to the background metric $b$. Moreover, this relation together with \eqref{limit1} yields 
\begin{equation}\label{limit111}
\lim_{s\to\infty}e^{3s}\left(w\circ F^{-1}-1\right)=e^{3f}\mathbf{w}_{\infty}\circ \Theta_{\infty}^{-1}=: w_{\infty},
\end{equation}
and all derivatives of this quantity decay at the rate $O(e^{-2s})$ when measured in $b$; the notation $\Theta_{\infty}$ indicates the limiting diffeomorphism of the torus from Proposition \ref{pro:longtime}.
	
\begin{lemma}\label{lemma34}
Under the hypothesis of Lemma \ref{ogoinoinhoih}, the extension manifold $(\Omega_+,g_+)$ is asymptotically hyperbolic.
\end{lemma}

\begin{proof}
It suffices to rearrange the expression for the metric in the asymptotic end. Observe that utilizing \eqref{qoignoqioginqh} and \eqref{limit111} produces 
\begin{align}\label{AHg_+}
\begin{split}
(F^{-1})^*g_+=& (F^{-1})^*\left(w^2dt^2+\gamma_t\right)\\
=&(F^{-1})^*\left(dt^2+\gamma_t+(w^2-1)dt^2 \right)\\
=&(F^{-1})^*\left(F^*b+(w^2-1)dt^2\right)\\
=&ds^2+e^{2s}\sigma+2(w\circ F^{-1}-1)d(t\circ F^{-1})^2+O(e^{-6s})d(t\circ F^{-1})^2\\
=&ds^2+e^{2s}\sigma+2w_{\infty}e^{-3s}ds^2-4w_{\infty}e^{-3s}f_idsd\theta^i+2w_{\infty}e^{-3s}f_if_jd\theta^id\theta^j+\mathcal{Q}\\
=&(1+w_{\infty}e^{-3s})^2ds^2+e^{2s}\sigma+\mathcal{Q},
\end{split}
\end{align}
where 
\begin{equation}\label{alk3nhapgaj1}
|\mathcal{Q}|_{b}+| \pmb\nabla \mathcal{Q}|_b+|\pmb\nabla^2 \mathcal{Q}|_b=O(e^{-4s}).
\end{equation}
Consider now the change of radial coordinate $r=s-\frac{1}{3}w_{\infty}e^{-3s}$ and note that this yields
\begin{equation}\label{ooinoqinoinih}
\psi^*g_+=dr^2+\left(e^{2r}+\frac{2w_\infty}{3}e^{-r}\right)\sigma+\mathcal{Q},
\end{equation}
in which $\mathcal{Q}$ satisfies \eqref{alk3nhapgaj1} with $s$ replaced by $r$. Here $\psi:(r_0,\infty)\times T^2 \rightarrow \Omega_+ \setminus \mathcal{K}$, where $\mathcal{K}$ is a compact set, is the diffeomorphism arising from the composition of $F^{-1}$ and the radial coordinate change.
The desired result now follows.
\end{proof}

We end this section by recording an observation that the normal flow gives rise to a monotonic static Brown-York mass which compares the geometry of leaves in the Kottler setting to that in the Bartnik-Shi-Tam extension.  The next result follows from Lu-Miao \cite[Proposition 2.2]{LuMiao}, after noting that the Kottler manifold $M_k$ is static with potential $V=e^s$. 
	
\begin{prop}\label{propmontonicity}
Let $\{(\Sigma_t,\gamma_t)\}_{t=0}^{\infty}$ denote the normal flow emanating from a surface $\Sigma_0\subset M_k$ having Gauss curvature $K_0 >-1$. If $(\Omega_+,g_+=w^2 dt^2+\gamma_t)$ is a corresponding Bartnik-Shi-Tam extension then
\begin{equation}
\frac{d}{dt}\int_{\Sigma_t}V\left(H_t-H_{+}\right)dA_{\gamma_t}=-\int_{\Sigma_t}w^{-1}(w-1)^{2}\left(H_t\partial_{\nu}V+\frac{1}{2}V\left(H_t^2-|h_t|_{\gamma_t}^2\right)\right)dA_{\gamma_t},
\end{equation}
where $H_t$ and $H_+$ are the mean curvatures of $\Sigma_t$ with respect to $b$ and $g_+$ respectively, while $h_t$ is the second fundamental form with respect to $b$, and $\nu$ is the unit outer normal. In particular, since $\partial_{\nu}V>0$ and $H_t^2-|h_t|_{\gamma_t}^2>0$ the static Brown-York mass is nonincreasing in $t$, and is  constant if and only if $w=1$.
\end{prop}
	
\section{Convergence of Static Brown-York Mass in the Extension}\label{Sec:asymptotes}
\setcounter{equation}{0}
\setcounter{section}{4}

A Riemannian 3-manifold $(M,g)$ will be called \textit{asymptotically hyperbolic with toroidal infinity}, if there is a compact set $\mathcal{K}\subset M$ such that its complement is diffeomorphic to a cylinder with torus cross-sections, and in the coordinates given by the diffeomorphism $\psi:(r_0,\infty)\times T^2 \rightarrow M\setminus\mathcal{K}$ the metric has an expansion
\begin{equation}\label{oiaoginoaqinh}
\psi^* g=dr^2+e^{2r} \sigma+e^{-r}\mathbf{m} +Q, 
\end{equation}
where $r\in (r_0,\infty)$ is the radial coordinate, $\sigma$ is a flat metric and $\textbf{m}$ is a symmetric 2-tensor all on $T^2$, and $Q$ is a symmetric 2-tensor on $(r_0,\infty)\times T^2$ with the property that
\begin{equation}\label{alk3nhapgaj1111}
|Q|_{b}+|\pmb\nabla Q|_b+|\pmb\nabla^2 Q|_b=o(e^{-3r}).
\end{equation}
As before $b$ is the model hyperbolic metric $dr^2+e^{2r}\sigma$ on $(r_0,\infty)\times T^2$, and $\pmb\nabla$ is the Levi-Civita connection of $b$. The quantity $\mathrm{Tr}_{\sigma}(3\mathbf{m})$ is referred to as the \textit{mass aspect function}, and by Chru\'{s}ciel-Herzlich \cite{ChruscielHerzlich} it gives rise to a well-defined total mass 
\begin{equation}
m(g)=\frac{1}{16\pi}\int_{T^2}\mathrm{Tr}_{\sigma}(3\mathbf{m}) dA_{\sigma}
\end{equation}
if the scalar curvature satisfies $r (R_g +6)\in L^{1}(M\setminus\mathcal{K})$; this condition will henceforth be included in the asymptotics definition above.
We have that the Bartnik-Shi-Tam extension $(\Omega_+,g_+)$ constructed in the previous section is asymptotically hyperbolic with toroidal infinity by Lemma \ref{lemma34}. 
Therefore, the extension's total mass may be computed from \eqref{ooinoqinoinih} to yield 
\begin{equation}\label{foiqbifubiuq}
m(g_+)=\frac{1}{4\pi}\int_{T^2}w_{\infty}dA_{\sigma},
\end{equation}
after noting that $\mathbf{m}=\tfrac{2}{3}w_\infty\sigma$ in this case.

\begin{prop}\label{proposition}
Under the hypotheses and notation of Proposition \ref{propmontonicity} it holds that
\begin{equation}
\frac{1}{8\pi}\int_{\Sigma_0}VH_0(1-w_0^{-1})dA_{\gamma_0}\geq m(g_+),
\end{equation}
where $w_0>0$ is the initial condition of equation \eqref{PDEu}.
\end{prop}

\begin{proof}
We first show that the static Brown-York mass of normal flow leaves in the extension converges to the total mass. To see this recall that
\begin{equation}
H_+= w^{-1} H_t, \quad H_t=2+O(e^{-2t}),\quad e^{3s}(w-1)=w_{\infty}+o(1),\quad dA_{\gamma_t}=e^{2s}(1+O(e^{-2s}))dA_{\sigma},
\end{equation}
where $s$ and $t$ are related through \eqref{qoignoqioginqh}. It follows that
\begin{align}\label{eqm2}
\begin{split}
\int_{\Sigma_t}V\left(H_t-H_{+}\right)dA_{\gamma_t}=&\int_{T^2}e^s (2+O(e^{-2t}))\left(1-w^{-1}\right) e^{2s}(1+O(e^{-2s}))dA_{\sigma}\\
=&\int_{T^2}2w_\infty  dA_{\sigma} +o(1).
\end{split}
\end{align} 
Combining this with monotonicity of the static Brown-York mass from Proposition \ref{propmontonicity}, along with formula \eqref{foiqbifubiuq} for the mass, produces
\begin{equation}
\frac{1}{8\pi}\int_{\Sigma_0}VH_0(1-w_0^{-1})dA_{\gamma_0}
\geq \lim_{t\rightarrow\infty}\frac{1}{8\pi}\int_{\Sigma_t}V\left(H_t-H_{+}\right)dA_{\gamma_t}= m(g_+).
\end{equation}
\end{proof}

\section{Proof of the Main Theorem}\label{Sec:mainthm}
\setcounter{equation}{0}
\setcounter{section}{5}

The primary result will be proven with the help of a positive mass theorem involving corners. This version of the positive mass theorem in the setting of asymptotically hyperbolic manifolds with toroidal infinity, may be
established using the level set technique associated with spacetime harmonic functions. This approach has led to the desired theorem when corners are not present \cite[Theorems 1.1 and 1.2]{AlaeeKhuriHung}, and Tsang \cite[Theorem 1.1]{Tsang} has shown that it gives rise to a theorem with corners for the asymptotically flat case. 
On a Riemannian 3-manifold a function will be referred to as \textit{spacetime harmonic} \cite{HKK} if it satisfies the equation
\begin{equation}\label{spacetimeharmonic}
\Delta u-3|\nabla u|=0.
\end{equation}
The next lemma is a consequence of \cite[Theorems 1.1, 1.2, and 1.3]{AlaeeKhuriHung} when $k=-g$, while taking into account the analysis at corners presented in \cite{Tsang}.

\begin{lemma}\label{goqaoingoinh}
Let $(M,g)$ be an orientable 3-dimensional asymptotically hyperbolic manifold with toroidal infinity
and having nonempty boundary $\partial M\neq \varnothing$. Suppose that $M$ satisfies the homotopy condition with respect to conformal infinity, and $H_2(M,\partial M;\mathbb{Z})=0$. Assume further that there is a closed smooth hypersurface `corner' $\mathcal{S}\subset M$ disjoint from the boundary, which separates $M$ into a compact portion $M_-$ and a noncompact portion $M_+$ containing the asymptotic end,
such that $g$ is smooth up to $\mathcal{S}$ and is Lipschitz across it.
\begin{enumerate}
\item If the boundary $\partial M$ has mean curvature $H\leq 2$ with respect to the inner normal, then 
there exists a spacetime harmonic function $u\in C^{1,\alpha}_{loc}(M)\cap C^{2,\alpha}_{loc}(M\setminus\mathcal{S})$, $0<\alpha<1$ which asymptotes to the radial coordinate function $e^r$, and induces the mass lower bound
\begin{equation}
m(g)\geq\frac{1}{16\pi}\int_{M\setminus \mathcal{S}}\left(\frac{|\nabla^2 u-3|\nabla u|g|^2}{|\nabla u|}+(R_g +6)|\nabla u|\right)d\mathbf{V} +\frac{1}{8\pi}\int_{\mathcal{S}}(H_- -H_+)dA,
\end{equation}
where $H_{\pm}$ are the mean curvatures of $\mathcal{S}\subset M_{\pm}$ with respect to the normal pointing towards the asymptotic end. In particular, if $R_g \geq -6$ then the mass is nonnegative $m(g)\geq 0$.

\item The same conclusion holds if the boundary $\partial M$ contains additional components of genus zero with $H\leq -2$.

\item Under the same hypotheses together with $R_g\geq-6$, if $m(g)= 0$ then $(M,g)$ is isometric to a region $\left([s_0,\infty)\times T^2,b\right)$ in a Kottler manifold.

\item If $\partial M$ is a minimal surface and $\partial_1 M$ is the component with least area then a Penrose-type inequality holds
\begin{equation}
m(g)\geq \mathcal{C}\frac{|\partial_1 M|}{16\pi},
\end{equation}
where $\mathcal{C}=4\min_{\partial M}\partial_{\upsilon}u>0$ and $\upsilon$ denotes the unit inner normal.
\end{enumerate}    
\end{lemma}

We may now prove Theorem \ref{thm:PMT}.  Since the boundary component $\Sigma$ of $(\Omega,g)$ has Gauss curvature $K>-1$ and admits a graphical isometric embedding $\Sigma_0$ into a Kottler manifold, Lemmas \ref{ogoinoinhoih} and \ref{lemma34} can be used to construct and asymptotically hyperbolic extension $(\Omega_+,g_+)$ with toroidal infinity. Moreover, since the mean curvature of $\Sigma$ is positive $H>0$, we may choose the initial condition of \eqref{PDEu} to be $w_0=H_0/H$ where $H_0$ is the mean curvature of $\Sigma_0$, so that the mean curvature of $\partial\Omega_+=\Sigma_0$ agrees with that of $\Omega$, that is $H_+=H$. The remaining hypotheses for $(\Omega,g)$ are sufficient to apply the positive mass theorem with corners result, Lemma \ref{goqaoingoinh}, to the composite manifold $(\Omega \cup \Omega_+,g\cup g_+)$. Note that the homotopy condition with respect to conformal infinity is satisfied, since $\Omega$ satisfies the homotopy condition with respect to $\Sigma$, and the extension $\Omega_+$ is homotopy equivalent to $\Sigma_0$. It follows that $m(g_+)\geq 0$, and hence Proposition \ref{proposition} implies that the static Brown-York mass is nonnegative $m_{BY}^S(\Sigma)\geq 0$. This establishes parts (1) and (2) of the main theorem. Parts (3) and (4) follow similarly from the same numbered portions of Lemma \ref{goqaoingoinh}.

\appendix

\section{An Example}
\label{appA}

Here we show that results of Theorem \ref{thm:PMT} fail to hold if some hypotheses are not satisfied.
Consider $\Omega=[r_h,r_0]\times T^2$ with Riemannian metric
\begin{equation}
g=r^{-2}(1-r^{-3})^{-1}dr^2+r^2(1-r^{-3})d\xi^2+r^2d\theta^2,
\end{equation}
where $r_h\geq 1$, and $\xi$, $\theta$ are coordinates on $T^2$ with periods $P_\xi$ and $P_\theta$ respectively. Note that if $P_\xi=4\pi/3$ then this metric can be extended smoothly to $r=1$, and in this case with $r_0$ replaced by $\infty$ the resulting manifold is called the \textit{Horowitz-Myers geon} \cite{Horowitz,Woolgar}. The geon has the topology of a solid torus $S^1 \times D^2$, is static with potential $V=r$ and scalar curvature $R_g=-6$, and is asymptotically hyperbolic with toroidal infinity
having total mass $-\tfrac{1}{16\pi}P_{\xi}P_{\theta}$. Let $\Sigma=\{r_0\}\times T^2$ and $\Sigma_h=\{r_h\}\times T^2$, $r_h >1$ be the boundary components of $\Omega$, and consider the Kottler manifold $M_k=((0,\infty)\times T^2,b)$ with metric
\begin{equation}
b=r^{-2}dr^2+r^2 \sigma,\quad\quad\quad \sigma=(1-r_0^{-3})d\xi^2+d\theta^2.
\end{equation}
Clearly $\Sigma$ isometrically embeds into $M_k$ as the surface $\Sigma_0=\{r=r_0\}$. Straightforward calculations show that the mean curvatures and area element are given by
\begin{equation}
H_0 =2, \quad\quad H=(1-r_0^{-3})^{-1/2}\left(2-\frac{1}{2}r_0^{-3}\right)=2+\frac{1}{2}r_0^{-3}+O(r_0^{-6}),\quad\quad dA=r_0^2 \sqrt{1-r_0^{-3}}d\xi d\theta,
\end{equation}
and therefore the static Brown-York mass becomes
\begin{equation}
m_{BY}^S(\Sigma)=\frac{1}{8\pi}\int_{\Sigma}V(H_0-H)dA=-\frac{1}{16\pi}P_{\xi}P_{\theta}+O(r_0^{-3}).
\end{equation}
We conclude that this quasi-local mass is negative for large $r_0$. However, $(\Omega,g)$ violates the mean curvature hypothesis of the main theorem since $H>2$ at $\Sigma_h$. Alternatively, one may consider the case with $r_h=1$ to obtain a further counterexample, when there is no inner boundary and $\Omega$ does not satisfy the homotopy condition.


\begin{thebibliography}{99}

\bibitem{AlaeeKhuriHung} A. Alaee, P.-K. Hung, and M. Khuri, \textit{The positive energy theorem for asymptotically hyperboloidal initial data sets with toroidal infinity and related rigidity results}, Comm. Math. Phys. \textbf{396} (2022), no. 2, 451--480.



\bibitem{AlaeeKhuriYau3} A. Alaee, M. Khuri, and S.-T. Yau, \textit{Geometric inequalities for quasi-local masses}, Comm. Math. Phys., \textbf{378} (2020), no. 1, 467--505.

\bibitem{AlaeeKhuriYau2} A. Alaee, M. Khuri, and S.-T. Yau, \textit{A quasi-local mass}, Comm. Math. Phys., to appear. arxiv 2309.02770.

\bibitem{AlaeeLesourdYau} A. Alaee, M. Lesourd, and S.-T. Yau, A localized spacetime Penrose inequality and horizon detection with quasilocal mass, J. Differential Geom., \textbf{125} (2023), no.3, 405--425.

\bibitem{Bartnik} R. Bartnik, \textit{Quasi-spherical metrics and prescribed scalar curvature},
J. Differential Geom., \textbf{37} (1993), no. 1, 31--71.

\bibitem{BrownYork} J. Brown, and J. York, \textit{Quasilocal energy and conserved charges derived from the gravitational action}, Phys. Rev. D., \textbf{47}, no. 4, (1993), 1407--1419.

\bibitem{ChruscielHerzlich} P. Chru\'{s}ciel, and M. Herzlich, \textit{The mass of asymptotically hyperbolic Riemannian manifolds}, Pacific J. Math., \textbf{212} (2003), no. 2, 231--264.

\bibitem{ChruscielSimon} P. Chru\'{s}ciel, and W. Simon, \textit{Towards the classification of static vacuum spacetimes with negative cosmological constant}, J. Math. Phys., \textbf{42} (2001), no. 4, 1779--1817.

\bibitem{EGM} M. Eichmair, G. Galloway, and A. Mendes, \textit{Initial data rigidity results},
Comm. Math. Phys., \textbf{386} (2021), no. 1, 253--268.

\bibitem{EMW} M. Eichmair, P. Miao, and X. Wang, \textit{Extension of a theorem of Shi and Tam},
Calc. Var. Partial Differential Equations, \textbf{43} (2012), no. 1-2, 45--56.

\bibitem{Gromov1} M. Gromov, \textit{Partial Differential Relations},  
Ergeb. Math. Grenzgeb., \textbf{9} [Results in Mathematics and Related Areas],
Springer-Verlag, Berlin, 1986.

\bibitem{Gromov} M. Gromov, \textit{Four lectures on scalar curvature}, World Scientific Publishing, Hackensack, NJ, 2023, 1--514. 

\bibitem{HKK} S. Hirsch, D. Kazaras, and M. Khuri, \textit{Spacetime harmonic functions and the mass of 3-dimensional asymptotically flat initial data for the Einstein equations}, J. Differential Geom., \textbf{122} (2022), no. 2, 223--258.

\bibitem{Horowitz} G. Horowitz, and R. Myers, \textit{The AdS/CFT correspondence and a new positive energy conjecture for general relativity}, Phys. Rev. D, \textbf{59} (1999), 1--12.

\bibitem{KS} N. Krylov, and M. Safonov, \textit{A certain property of solutions of
parabolic equations with measurable coefficients}, Mathematics of
the USSR-Izvestiya, \textbf{16} (1981), no. 1, 151--164.

  
\bibitem{LiuYau} C.-C. Liu, and S.-T. Yau, \textit{Positivity of quasi-local mass II}, J. Amer. Math. Soc., \textit{19} (2006), no. 1, 181--204.

  
\bibitem{LuMiao} S. Lu, and P. Miao, \textit{Minimal hypersurfaces and boundary behavior of compact manifolds with nonnegative scalar curvature}, J. Differential Geom., \textbf{113} (2019), no. 3, 519--566.




\bibitem{PenroseR} R. Penrose, \textit{Some unsolved problems in classical general relativity}, Annals of Mathematics Studies, \textbf{102} (1982), 631--668.

\bibitem{ShiTam} Y. Shi, and L.-F. Tam, \textit{Positive mass theorem and the boundary behaviors of compact manifolds with nonnegative scalar curvature}, J. Differential Geom., \textbf{62} (2002), no. 1, 79--125.

\bibitem{ShiTam2007} 
Y. Shi, and L.-F. Tam, 
\textit{Rigidity of compact manifolds and positivity of quasi-local mass}, Classical Quantum Gravity, \textbf{24} (2007), no. 9, 2357--2366.

\bibitem{SWW} Y. Shi, W. Wang, and G. Wei, \textit{Total mean curvature of the boundary and nonnegative scalar curvature fill-ins}, J. Reine Angew. Math., \textbf{784} (2022), 215--250.

\bibitem{Tsang} T.-Y. Tsang, \textit{On a spacetime positive mass theorem with corners}, preprint, 2022.
arXiv:2109.11070

\bibitem{Wang} M.-T. Wang, \textit{Quasi-local mass and isometric embedding with reference to a static spacetime}, in ``The Role of Metrics in the Theory of Partial Differential Equations", Advanced Studies in Pure Mathematics, \textbf{85} (2020), 453--462.

\bibitem{WangYau2007} M.-T. Wang, and S.-T. Yau, \textit{A generalization of Liu-Yau's quasi-local mass}, Comm. Anal. Geom., \textbf{15} (2007), no. 2, 249--282.

\bibitem{WangYau} M.-T. Wang, and S.-T. Yau, \textit{Isometric embeddings into the Minkowski space and new quasi-local mass}, Comm. Math. Phys., \textbf{288} (2009), no. 3, 919--942.

\bibitem{Woolgar} E. Woolgar, \textit{The rigid Horowitz-Myers conjecture}, J. High Energy Phys., (2017), no. 3, Art. 104.
  
\bibitem{Yau} S.-T. Yau, \textit{Geometry of three manifolds
and existence of black hole due to boundary effect}, Adv. Theor. Math. Phys., \textbf{5} (2001), no.4, 
755--767.

\end{thebibliography}
\end{document}